\documentclass[12pt]{amsart}
\usepackage{amscd,amssymb,latexsym}
\usepackage{fullpage}
\DeclareMathOperator*{\hocolimit}{\mathrm{holim}}
\newcommand{\hcl}[1]{{\displaystyle
            \hocolimit_{\substack{\hbox to 20pt{\rightarrowfill} \\ #1}}\,}}

\newcommand{\Fbarbar}{\overline{\overline{F}}}
\newcommand{\Fbar}{\overline{F}}
\newcommand{\Gbar}{\overline{G}}
\newcommand{\Kbar}{\overline{K}}
\newcommand{\tensorQ}{\otimes_Q}
\newcommand{\Jac}{\mathrm{J}}
\newcommand{\qcoh}{\mathrm{qCoh}}
\newcommand{\coh}{\mathrm{coh}}
\newcommand{\Inj}{\mathrm{Inj}}
\newcommand{\StMod}{\mathrm{StMod}}
\newcommand{\stmod}{\mathrm{stmod}}
\newcommand{\growth}{\mathrm{growth}}

\newcommand{\fg}{\mathcal{FG}}
\newcommand{\chivec}{\mathbf{\chi}}
\newcommand{\HOmega}{H^{\Omega}}
\newcommand{\omegabar}{\overline{\omega}}

\newcommand{\Jbar}{\overline{J}}
\newcommand{\Jbarbar}{\overline{\overline{J}}}

\newcommand{\ass}{\mathrm{Ass}}

\newcommand{\Mdef}[2]{\newcommand{#1}{\relax \ifmmode #2 \else $#2$\fi}}

\newcommand{\thick}{\mathrm{thick}}
\newcommand{\loc}{\mathrm{loc}}
\newcommand{\len}{\mathrm{len}}
\newcommand{\finbuilds}{\models}
\newcommand{\builds}{\vdash}

\newcommand{\tensor}{\otimes}

\newcommand{\sdr}{\rtimes}
\newcommand{\map}{\mathrm{map}}
\newcommand{\Hom}{\mathrm{Hom}}
\newcommand{\Tor}{\mathrm{Tor}}
\newcommand{\Ext}{\mathrm{Ext}}
\Mdef{\bhom}{\mathbf{\hat{H}om}}

\Mdef{\Mod}{\mathrm{mod}}

\newcommand{\st}{\; | \;}

\newtheorem{thm}{Theorem}[section]
\newtheorem{lemma}[thm]{Lemma}
\newtheorem{prop}[thm]{Proposition}
\newtheorem{cor}[thm]{Corollary}
\newtheorem{wish}[thm]{Wish}

\theoremstyle{definition}

\newtheorem{defn}[thm]{Definition}

\newtheorem{example}[thm]{Example}

\newtheorem{remark}[thm]{Remark}

\newcommand{\qqed}{\qed \\[1ex]}
\renewenvironment{proof}[1][\hspace*{-.8ex}]{\noindent {\bf Proof #1:\;}}{\qqed}

\newcommand{\cosusp}[1]{\Sigma_{#1}}

\Mdef{\PH} {\Phi^H}
\Mdef{\PK} {\Phi^K}
\Mdef{\PL} {\Phi^L}
\Mdef{\PT} {\Phi^{\T}}

\Mdef{\ef}{E{\cF}_+}
\Mdef{\etf}{\widetilde{E}{\cF}}
\Mdef{\eg}{E{G}_+}
\Mdef{\etg}{\tilde{E}{G}}
\Mdef{\infl}{\mathrm{inf}}
\Mdef{\defl}{\mathrm{def}}
\Mdef{\res}{\mathrm{res}}
\Mdef{\ind}{\mathrm{ind}}

\Mdef{\univ}{\mathcal{U}}

\Mdef{\Fp}{\mathbb{F}_p}
\Mdef{\Zpinfty}{\Z /p^{\infty}}
\Mdef{\Zpadic}{\Z_p^{\wedge}}

\newcommand{\bi}{\begin{itemize}}
\newcommand{\be}{\begin{enumerate}}
\newcommand{\bc}{\begin{center}}
\newcommand{\bd}{\begin{description}}
\newcommand{\ei}{\end{itemize}}
\newcommand{\ee}{\end{enumerate}}
\newcommand{\ec}{\end{center}}
\newcommand{\ed}{\end{description}}

\newcommand{\trichotomy}[3]{\left\{ \begin{array}{ll}#1\\#2\\#3 \end{array}\right.}

\newcommand{\lra}{\longrightarrow}
\newcommand{\lla}{\longleftarrow}

\Mdef{\we}{\mathbf{we}}
\Mdef{\fib}{\mathbf{fib}}
\Mdef{\cof}{\mathbf{cof}}
\Mdef{\BI}{\mathcal{BI}}

\newcommand{\ann}{\mathrm{ann}}
\newcommand{\fibre}{\mathrm{fibre}}
\newcommand{\cofibre}{\mathrm{cofibre}}

\Mdef{\A}{\mathbb{A}}
\Mdef{\B}{\mathbb{B}}
\Mdef{\C}{\mathbb{C}}
\Mdef{\D}{\mathbb{D}}
\Mdef{\E}{\mathbb{E}}
\Mdef{\T}{\mathbb{T}}
\Mdef{\F}{\mathbb{F}}
\Mdef{\G}{\mathbb{G}}
\Mdef{\I}{\mathbb{I}}
\Mdef{\N}{\mathbb{N}}
\Mdef{\Q}{\mathbb{Q}}
\Mdef{\R}{\mathbb{R}}
\Mdef{\bbS}{\mathbb{S}}
\Mdef{\Z}{\mathbb{Z}}

\Mdef{\bA}{\mathbb{A}}
\Mdef{\bB}{\mathbb{B}}
\Mdef{\bC}{\mathbb{C}}
\Mdef{\bD}{\mathbb{D}}
\Mdef{\bE}{\mathbb{E}}
\Mdef{\bF}{\mathbb{F}}
\Mdef{\bG}{\mathbb{G}}
\Mdef{\bH}{\mathbb{H}}
\Mdef{\bI}{\mathbb{I}}
\Mdef{\bJ}{\mathbb{J}}
\Mdef{\bK}{\mathbb{K}}
\Mdef{\bL}{\mathbb{L}}
\Mdef{\bM}{\mathbb{M}}
\Mdef{\bN}{\mathbb{N}}
\Mdef{\bO}{\mathbb{O}}
\Mdef{\bP}{\mathbb{P}}
\Mdef{\bQ}{\mathbb{Q}}
\Mdef{\bR}{\mathbb{R}}
\Mdef{\bS}{\mathbb{S}}
\Mdef{\bT}{\mathbb{T}}
\Mdef{\bU}{\mathbb{U}}
\Mdef{\bV}{\mathbb{V}}
\Mdef{\bW}{\mathbb{W}}
\Mdef{\bX}{\mathbb{X}}
\Mdef{\bY}{\mathbb{Y}}
\Mdef{\bZ}{\mathbb{Z}}

\Mdef{\cA}{\mathcal{A}}
\Mdef{\cB}{\mathcal{B}}
\Mdef{\cC}{\mathcal{C}}
\Mdef{\mcD}{\mathcal{D}} % Something funny about \cD.
\Mdef{\cE}{\mathcal{E}}
\Mdef{\cF}{\mathcal{F}}
\Mdef{\cG}{\mathcal{G}}
\Mdef{\mcH}{\mathcal{H}} % There's something funny about \cH: it 
\Mdef{\cI}{\mathcal{I}}
\Mdef{\cJ}{\mathcal{J}}
\Mdef{\cK}{\mathcal{K}}
\Mdef{\mcL}{\mathcal{L}}% There's something funny about \cL: it 
\Mdef{\cM}{\mathcal{M}}
\Mdef{\cN}{\mathcal{N}}
\Mdef{\cO}{\mathcal{O}}
\Mdef{\cP}{\mathcal{P}}
\Mdef{\cQ}{\mathcal{Q}}
\Mdef{\mcR}{\mathcal{R}}% There's something funny about \cR: it 
\Mdef{\cS}{\mathcal{S}}
\Mdef{\cT}{\mathcal{T}}
\Mdef{\cU}{\mathcal{U}}
\Mdef{\cV}{\mathcal{V}}
\Mdef{\cW}{\mathcal{W}}
\Mdef{\cX}{\mathcal{X}}
\Mdef{\cY}{\mathcal{Y}}
\Mdef{\cZ}{\mathcal{Z}}

\Mdef{\tA}{\tilde{A}}
\Mdef{\tB}{\tilde{B}}
\Mdef{\tC}{\tilde{C}}
\Mdef{\tE}{\tilde{E}}
\Mdef{\tH}{\tilde{H}}
\Mdef{\tK}{\tilde{K}}
\Mdef{\tL}{\tilde{L}}
\Mdef{\tM}{\tilde{M}}
\Mdef{\tN}{\tilde{N}}
\Mdef{\tP}{\tilde{P}}

\Mdef{\tf}{\tilde{f}}

\Mdef{\Ab}{\overline{A}}
\Mdef{\Bb}{\overline{B}}
\Mdef{\Cb}{\overline{C}}
\Mdef{\Db}{\overline{D}}
\Mdef{\Eb}{\overline{E}}
\Mdef{\Hb}{\overline{H}}
\Mdef{\Gb}{\overline{G}}
\Mdef{\Ib}{\overline{I}}
\Mdef{\Kb}{\overline{K}}
\Mdef{\Lb}{\overline{L}}
\Mdef{\Mb}{\overline{M}}
\Mdef{\Nb}{\overline{N}}
\Mdef{\Qb}{\overline{Q}}
\Mdef{\Tb}{\overline{T}}

\Mdef{\db}{\overline{d}}
\Mdef{\hb}{\overline{h}}
\Mdef{\qb}{\overline{q}}
\Mdef{\rb}{\overline{r}}
\Mdef{\tb}{\overline{t}}
\Mdef{\ub}{\overline{u}}
\Mdef{\vb}{\overline{v}}

\Mdef{\hc}{\hat{c}}
\Mdef{\he}{\hat{e}}
\Mdef{\hf}{\hat{f}}
\Mdef{\hA}{\hat{A}}
\Mdef{\hH}{\hat{H}}
\Mdef{\hJ}{\hat{J}}
\Mdef{\hM}{\hat{M}}
\Mdef{\hP}{\hat{P}}
\Mdef{\hQ}{\hat{Q}}

\Mdef{\thetab}{\overline{\theta}}
\Mdef{\phib}{\overline{\phi}}

\Mdef{\uA}{\underline{A}}
\Mdef{\uB}{\underline{B}}
\Mdef{\uC}{\underline{C}}
\Mdef{\uD}{\underline{D}}

\Mdef{\bolda}{\mathbf{a}}
\Mdef{\boldb}{\mathbf{b}}
\Mdef{\boldD}{\mathbf{D}}

\Mdef{\fm}{\frak{m}}
\Mdef{\eps}{\epsilon}

\input{xypic}

\begin{document}
\title{Complete intersections and derived categories.}

\author{D.J.Benson}
\address{Department of Mathematics, University of Aberdeen, Aberdeen AB24 3UE,
UK}
\email{bensondj@maths.abdn.ac.uk}

\author{J.P.C.Greenlees}
\address{School of Mathematics and Statistics, Hicks Building, 
Sheffield S3 7RH, UK}
\email{j.greenlees@sheffield.ac.uk}
\date{}

\begin{abstract}
We propose a definition of when a triangulated category should be
considered a complete intersection. We show that for the derived category 
of a complete local Noetherian commutative ring $R$,  
the condition on the derived
category $D(R)$ holds precisely when $R$ is 
a complete intersection in the classical sense. 
\end{abstract}

\thanks{The research was partially supported by EPSRC Grant number
EP/E012957/1. }
\maketitle

%\tableofcontents

\section{Introduction}
\label{sec:Intro}
\subsection{The context.}
In algebraic geometry, the best behaved varieties are complete intersections, 
which are subvarieties of affine space specified by the right number 
of equations in the sense that 
if they are of codimension $c$ then only $c$ equations are required. 
Considering this locally, we obtain an analogue for
 commutative local Noetherian rings $R$, where the corresponding 
condition is that 
$$R=Q/(f_1,\ldots, f_c)$$
for a regular local ring $Q$ by a regular sequence, $f_1,f_2, \ldots, f_c$
The smallest possible value of $c$ (as $Q$ and the regular sequence vary)
is called the {\em codimension} of $R$. In commutative algebra the terminology
is slightly different, so that $R$ is  a {\em complete intersection (ci)} 
if its completion is of the stated form.

We recall below that if $R$ is ci of codimension $c$, one may construct
a resolution of any finitely generated module growing like a polynomial 
of degree $c-1$. In particular,  the ring $\Ext_R^*(k,k)$ has polynomial growth
(we say that $R$ is {\em gci}).
Perhaps the most striking result about ci rings is the theorem of Gulliksen 
\cite{Gulliksen:1971a} 
which states that the ci and gci conditions are equivalent. 

In fact one may go further and show that the resolutions are constructed
in an eventually multi-periodic fashion. In particular, 
for a ci local ring $\Ext_R^*(k,k)$ is finite as a module over a 
commutative Noetherian ring of operators \cite[3.1]{Gulliksen:1974a}. 
This in turn opens the way to 
the theory of support varieties for modules over a ci ring \cite{Avramov/Buchweitz:2000a}.

\subsection{The aspiration.}
We are interested in extending the notion of ci rings to commutative
differential graded algebras (DGAs), commutative ring spectra 
\cite{Elmendorf/Kriz/Mandell/May:1997a, Hovey/Shipley/Smith:2000a} and 
related contexts. We have a particular interest in studying the cochains $C^*(BG)$ on the
classifying space of a finite group $G$, partly because of the consequences
for the cohomology ring $H^*(BG)$. We would like to follow the model of
\cite{Dwyer/Greenlees/Iyengar:2006a}, 
which shows that $C^*(BG)$ is Gorenstein in a homotopical 
sense for all finite groups $G$. This structural result
then establishes the existence of a local cohomology theorem for $H^*(BG)$, and in particular proves the
result of Benson and Carlson \cite{Benson/Carlson:1994a} 
that if $H^*(BG)$ is Cohen-Macaulay it is also Gorenstein. 

For $C^*(BG)$, the counterpart of the Ext algebra $\Ext_R^*(k,k)$ is 
the loop space homology of the $p$-completed classifying space, 
$H_*(\Omega(BG_p^{\wedge}))$. Of course if $G$ is a $p$-group, this is
simply the group ring $kG$ in degree 0. More generally, it is known 
to be of polynomial growth in many cases (Chevalley groups 
away from the defining characteristic, $A_4$ or $M_{11}$
in characteristic 2, ....). However
R.~Levi \cite{Levi:1995a,Levi:1996a,Levi:1997a,Levi:1999a} has
proved there is a dichotomy between small growth and large growth,
and given simple examples (semidirect products $(C_p \times C_p)\sdr C_3$
at $p\geq 5$) where the growth is exponential. 
Our aim is to give a version of the definition of ci,
which is homotopy invariant and applies to the module category; when 
it holds for $C^*(BG)$, it should imply that $\Omega(BG_p^{\wedge})$ has
 polynomial growth.

The purpose of the present paper is to lay foundations for these 
generalizations by giving a characterization of complete intersections 
in terms of the derived category $D(R)$ of $R$-modules. It will be apparent
that the definition applies to related triangulated categories such as
the homotopy category of $C^*(BG)$-modules. We explain briefly why 
existing definitions are not suitable for us. Firstly, although polynomial 
growth rate of the endomorphism ring of $k$ is an interesting and 
important property, its application is somewhat limited since 
it relies on having something like a Hilbert series as a measure of size. 
The difficulty with  Avramov's \cite{Avramov:1999a} characterization 
in terms of Andr\'e-Quillen homology (Quillen's conjecture) is rather
different. To start with, it cannot be stated 
in terms of the (additive) derived category, but more significant
for us is that it turns out that the topological Andr\'e-Quillen cohomology 
vanishes rather generally in the case of cochains on a space 
\cite{Mandell:2001a}, so that it includes far too many spaces. 

The topological and representation theoretic examples will be treated
elsewhere \cite{qzci, pzci, kGzci}.
%%We also intend give a more algebraic
%account of the special case of $C^*(BG)$ for a finite group $G$, building
%on the first author's representation theoretic approach to the loop space
%homology. 

\subsection{Organization of the paper.}

In Section \ref{sec:thedefinition} we will give a definition which 
applies to any triangulated category with suitable additional structure, 
and briefly describe some interesting classes of examples.

We then consider the particular example of the derived category of a commutative local ring, showing that our new definition does characterize ci rings. More
precisely,  in Section \ref{sec:statements} we outline our results for 
commutative
rings. Section \ref{sec:jac} makes some observations about the 
endomorphisms of the identity functor of the derived category. 
Sections \ref{sec:hypersurfaces1} and \ref{sec:ciiszci} 
describe the Avramov-Gulliksen
multiperiodic resolutions in a suitable form, showing that a ci ring 
satisfies our condition, and in Section \ref{sec:zciisgci} we establish the 
connection with the Ext algebra having polynomial growth rate.  
%In Part \ref{part:localalgebra}
% (Sections \ref{sec:jac} to \ref{sec:zciisgci}), 

%Parts
%\ref{part:cochains}
% (Sections \ref{sec:introcochains} to \ref{sec:hypersurfacespaces}), 

\subsection{Grading conventions.}
We will have cause to discuss homological and cohomological gradings.
Our experience is that this a frequent source of confusion, so we adopt
the following conventions. First, we refer to lower gradings as {\em degrees}
and upper gradings as {\em codegrees}. As usual, one may convert gradings to 
cogradings via the rule $M_n=M^{-n}$. Thus both chain complexes and cochain
complexes have differentials of degree $-1$ (which is to say, of codegree 1).
This much is standard. However, since we need to deal with both chain complexes
and cochain complexes it is essential to have separate notation for
homological suspensions ($\Sigma^i$) and  cohomological suspensions 
($\Sigma_i$): these are defined by 
$$(\Sigma^iM)_n=M_{n-i} \mbox{ and } (\Sigma_iM)^n=M^{n-i}.$$
Thus, for example, with reduced chains and cochains  of a based
space $X$,  we have
$$C_*(\Sigma^iX)=\Sigma^iC_*(X) \mbox{ and } C^*(\Sigma^iX)=\Sigma_iC^*(X).$$

\section{The definition.}
\label{sec:thedefinition}
Before giving the definition, we recall some standard technology and 
background results. 

\subsection{Terminology for triangulated categories.}
Recall that an object $X$ of a triangulated category $\cT$ is called
{\em small} if the natural map
$$\bigoplus_i [X,Y_i] \lra [X,\bigoplus_i Y_i]$$
is an isomorphism for any set of objects $Y_i$. 

A {\em thick} subcategory of 
$\cT$ is a full subcategory closed under completion of triangles and taking
retracts. We write $\thick (X)$ for the smallest thick subcategory containing
$X$, and if $A \in \thick (X)$ we also say 
`{\em $X$ finitely builds  $A$}' and write $X\finbuilds A$. 

A {\em localizing } subcategory of 
$\cT$ is a full subcategory closed under completion of triangles and taking
arbitrary coproducts. We write $\loc (X)$ for the smallest localizing
 subcategory containing
$X$, and if $A\in \loc (X)$ we also say `{\em $X$ builds $A$}' 
and write $X\builds A$.

Following \cite{Dwyer/Greenlees/Iyengar:2006b} 
we say that $X$ is {\em virtually small} if
$\thick (X)$ contains a non-trivial small object $W$, and we say that
any such $W$ is a {\em witness} for the fact that $X$ is virtually small.   

\subsection{Regularity.}
The archetype of a derived characterization  of a property of rings
is  the Auslander-Buchsbaum-Serre characterization of regularity:
a Noetherian local ring $R$ is regular if and only if 
the residue field has  a finite resolution by finitely generated
projectives; when this holds, every finitely generated module has 
a finite resolution by finitely generated
projectives. Since small objects are precisely those equivalent to 
 finite complexes of finitely generated projectives, this gives a 
characterization in terms of structure intrinsic to the
triangulated category together with a small amount of additional structure. 

We will rephrase these results in a style suggestive of our intentions
for complete intersections. First, 
since finitely generated modules are not characterized in terms
of the triangulated structure, we posit a triangulated subcategory 
$\fg$ of objects to play the corresponding role. In  the case of $D(R)$ 
we take the subcategory $\fg$ to consist of bounded complexes of 
finitely generated modules. 

We may package this in three definitions. The first is ideal-theoretic, 
the second refers to the rate of {\bf g}rowth of the derived endomorphism 
of the simple module $k$, and the third is in terms of the {\bf h}omotopy
theory of modules. 

\begin{defn}
(i) A local Noetherian ring $R$ is {\em regular} if the maximal ideal 
is generated by a regular sequence.

(ii) A local Noetherian ring $R$ is {\em g-regular} if $\Ext_R^*(k,k)$ is 
finite dimensional.

(iii) A local Noetherian ring $R$ is {\em h-regular} if every finitely generated
module is small in $D(R)$.
\end{defn}

It is not hard to see that g-regularity is equivalent to h-regularity or that
regularity implies g-regularity. Serre proved that g-regularity implies 
regularity, so the three conditions are equivalent.

We will provide a characterization of ci rings in the same style. 

\subsection{Work of Dwyer-Greenlees-Iyengar.}

We set the context by recalling the work  
of \cite{Dwyer/Greenlees/Iyengar:2006b} which led us to 
our definition. 
%It is designed for the context of $\cT =D(R)$ for a commutative 
%Noetherian ring $R$
%, so morally belongs in Part \ref{part:localalgebra}, 
%but we recall it here
%for expository purposes. 

\begin{defn} (\cite{Dwyer/Greenlees/Iyengar:2006b})
The ring $R$ is said to be a  {\em quasi-complete intersection (qci)}
if every object of $\fg$  is virtually small. 
\end{defn}

The reason for the terminology is that it was shown 
in \cite{Dwyer/Greenlees/Iyengar:2006b} 
that any ci local ring is qci, and conversely that every qci ring is 
Gorenstein. Examples were given to show that not all Gorenstein rings
are qci. 

The difficulty with the qci condition is that there is no control over how
a complex $X$ builds a non-trivial small object. It is therefore hard to 
deduce consequences about polynomial growth for example. 

\subsection{The centre of a triangulated category.}

To start with we suppose given a triangulated category $\cT$. The 
{\em centre} $Z\cT$ of $\cT$ is defined to be the graded ring of 
endomorphisms of the identity functor. This is a graded commutative
ring.% \cite{ZTcomm}. 

If $\cT $ is tensor triangulated, the endomorphism ring  of
the unit element gives elements of $Z\cT$, but typically there 
are many others. 

Given $\chi \in Z\cT$ of degree $a$,  for any object $X$, we may form 
the mapping cone $X/\chi$ of $\chi : \Sigma^a X \lra X$. This is
well defined up to non-unique equivalence. Indeed, given a map 
$f: X \lra Y$, the axioms of a triangulated category give
 a map $f:X/\chi \lra Y/\chi$ consistent with 
the defining triangles, but this is not usually unique or compatible
with composition. 

Now given a sequence of elements $\chi_1 , \chi_2, \ldots, \chi_n$
we may iterate this construction, and form 
$$K(X; \chivec) := X/\chi_1/\chi_2/\cdots /\chi_n, $$
which we sometimes refer to as the Koszul complex of the sequence. 

\begin{lemma}
Up to equivalence $K(X;\chivec)$ only depends on the sequence, and 
is independent of the order of the 
elements $\chi_i$.
\end{lemma}

\begin{proof}
For the first statement we argue by induction on $n$. We have already 
dealt with the case $n=1$. Suppose then that $n\geq 2$ and
$Y$ and $Y'$ are two models for $K(X; \chi_1, \ldots , \chi_{n-1})$. 
By induction there is an equivalence $e:Y \lra Y'$. Since $\chi_n$ is in 
the centre we have a commutative square
$$\begin{array}{ccc}
Y&\stackrel{\chi_n}\lra & Y\\
e \downarrow \simeq &&e\downarrow \simeq \\
Y'&\stackrel{\chi_n}\lra & Y',
\end{array}$$
and so the axioms provide an equivalence $Y/\chi_n \simeq Y'/\chi_n$ between
any choices of completions of the triangles. 

For the second statement, it suffices to deal with the case $n=2$, since
transpositions generate the symmetric group.  For the case $n=2$ we 
apply the octahedral axiom to the square
\begin{equation*}
\begin{array}{ccc}
X&\stackrel{\chi_1}\lra & X\\
\chi_2 \downarrow &&\downarrow \chi_2 \\
X&\stackrel{\chi_1}\lra & X.
\end{array}
\qedhere
\end{equation*}
\end{proof}

\subsection{Examples.}
We consider a triangulated category $\cT$, together with a triangulated
subcategy $\fg$ of objects we call {\em finitely generated}. We recommend
that readers concentrate on the first example, viewing the others as
generalizations to be investigated. We treat Example \ref{eg:localalgebra}
 in the present paper.  
%Part \ref{part:localalgebra}, 
We defer the investigations of Examples \ref{eg:cochains} and \ref{eg:squeezed}
to \cite{pzci} and \cite{kGzci}.
%in Parts 3 and 4 of the series. 
%\ref{part:cochains} and \ref{part:squeezed}. 

\begin{example}
\label{eg:localalgebra}
Our motivating example has $\cT =D(R)$ for a commutative Noetherian ring $R$. 
The small objects are those equivalent to a bounded complex of finitely 
generated
projectives.  The class  $\fg$ consists of the bounded complexes of finitely 
generated modules. 
\end{example}

\begin{example}
\label{eg:cochains}
The topological example has $\cT =Ho(C^*(X))$ for a space $X$.
%(further details will be given in Part 3
%\ref{part:cochains}, Section \ref{sec:introcochains}). 
The small objects include the modules  $C^*(Y)$ where $Y$ is a space with a 
map $f:Y\lra X$ and so that the homotopy fibre $F(f)$ has finite homology.
%(see Lemma \ref{lem:finfibissmall}). 
\end{example}

\begin{example}
\label{eg:squeezed}
An algebraic replacement for the category $Ho(C^*(BG))$
is given by taking  $\cT$ to be the $k$-cellularization of 
$K(\Inj \; kG)$ for a finite group $G$ as in \cite{Benson/Krause:2008a}. 
%We can identify the small
%and finitely generated objects by reference to $Ho(C^*(BG))$, but it
%would be better to be explicit. 
\end{example}

\begin{example}\label{eg:stable}
The category $\cT =\StMod ( kG)$ is the stable module category for a finite 
group $G$. The small objects consist of the subcategory $\stmod ( kG)$ 
of finitely generated modules. 

This suggests that all $G$ should be considered stably regular and that 
$\fg$ should again consist of finitely generated modules. 
\end{example}

\begin{example}
\label{eg:schemes}
For a Noetherian scheme $X$ we take the derived category $\cT =D(\qcoh (X))$  
of complexes of quasi-coherent sheaves on $X$. The small
objects include vector bundles (i.e., locally free sheaves).
We take  $\fg=D^b(\coh (X))$ to be the category of bounded complexes of
coherent sheaves. 
\end{example}

\subsection{The definition.}

Recall that if $S$ is a commutative ring, then the {\em Jacobson
radical} $\Jac (S)$ is defined to be the intersection of the 
maximal ideals. It is convenient to use the following well-known 
characterization. 

\begin{lemma}
We have $x \in \Jac (S)$ if and only if $1-xy$ is a unit for all
$y \in S$. \qed
\end{lemma}

Finally we may give the definition. 

\begin{defn}
We say that a triangulated category $\cT$ with a subcategory $\fg$ is 
{\em centrally ci (zci)} if there are elements $\chi_1, \ldots , \chi_c$
in $\Jac (Z\cT)$ so that for all objects $X$ of $\fg$, the Koszul complex
$K(X; \chivec)$ is small.  The smallest such $c$ is called the {\em codimension}
of $\cT$.
\end{defn}

First, we note that the definitions are arranged so that it is a 
tautology that if $\cT$ is h-regular, it is also zci of
codimension 0.  We refer to zci categories of codimension $1$ as 
{\em z-hypersurface} categories.

\section{Statement of results.}
\label{sec:statements}

For the rest of the present paper we restrict attention to the 
category $\cT=D(R)$ for a Noetherian local ring $R$, and $\fg$
denotes the subcategory of bounded complexes with finitely generated
homology.

\subsection{Complete intersections.}
In commutative algebra there are three styles for a definition of
a complete intersection ring: ideal theoretic, in terms of the growth of 
the Ext algebra and a derived version. 

\begin{defn}
(i) A local Noetherian ring $R$ is a {\em complete intersection (ci)} 
ring if its completion is of the form $Q/(f_1,f_2, \ldots , f_c)$ for
some regular ring $Q$ and some regular sequence $f_1, f_2, \ldots , f_c$.
The minimum such $c$ (over all $Q$ and regular sequences) is called
the {\em codimension} of $R$.

(ii) A local Noetherian ring $R$ is  {\em gci}  
if $\Ext_R^*(k,k)$ has polynomial growth. The  {\em g-codimension}
of $R$ is one more than the degree of the growth.

(iii) A local Noetherian ring $R$ is  {\em zci} if 
there are elements $z_1,z_2,\ldots z_c\in \Jac (ZD(R))$ so that 
$M/z_1/z_2/ \cdots /z_c$ is small for all complexes $M$ in $\fg$. 
The minimum such $c$ is called the  {\em z-codimension} of $R$.

%A local Noetherian ring $R$ is said to be  {\em bci} 
%if there is a regular ring $Q$ and map $Q\lra R$ so that 
%$R$ is virtually small as an $R^e$-module, where $R^e=R\tensor_QR$. 
\end{defn}

\begin{remark}
To avoid confusion we should comment on a trivial but awkward detail 
of terminology. 

Since the gci condition is insensitive to completion, 
the completion in the ci definition is necessary for 
the conditions to be equivalent for incomplete rings. It is also 
completely standard.  
On the other hand, it is inconvenient that there is no simple terminology 
for a local ring of the form $Q/(f_1, \ldots, f_c)$ for a regular ring 
$Q$ and a regular 
sequence $f_1,\ldots ,f_c$, and that the algebraic geometry terminology
has been subverted in commutative algebra. We have chosen to let the 
zci terminology correspond to the algebraic geometry conventions.
\end{remark}

The main result of the present paper is the following.

\begin{thm}
For a complete Noetherian local ring $R$, the ci condition, the gci condition and 
the zci condition are equivalent. The notions of codimension also agree. 
\end{thm}

The fact that ci implies zci amounts to the classical Avramov-Gulliksen
construction of eventually multiperiodic resolutions of modules over a 
ci ring. More precisely, because the construction is uniform, the construction 
can be lifted to Hochschild cohomology and hence provides elements of the 
centre of $D(R)$. This is described in Sections \ref{sec:hypersurfaces1}
and \ref{sec:ciiszci}. The fact that zci implies gci is a rather straightforward
estimate of growth given in Section \ref{sec:zciisgci}. Finally, the fact that
gci implies ci is Gulliksen's theorem.

\subsection{Hypersurfaces.}

In commutative algebra there are three styles for a definition of
a hypersurface ring: ideal theoretic, in terms of the growth of 
the Ext algebra and a derived version. In this case the second and
third styles have two variants each. 

\begin{defn}
(i) A local Noetherian ring $R$ is a {\em hypersurface} ring if $R=Q/(f)$ for
some regular ring $Q$ and non-zero element $f$.

(ii) A local Noetherian ring $R$ is a {\em g-hypersurface} ring 
if $\Ext_R^*(k,k)$ has polynomial growth of degree 
$\leq 1$. We say that $R$
 is a {\em p-hypersurface} ring if $\Ext_R^*(k,k)$ is eventually 
periodic. 

(iii) A local Noetherian ring $R$ is a {\em z-hypersurface} ring if 
there is an element $z\in \Jac (ZD(R))$ so that $M/z$ is small for all 
finitely generated modules $M$. We say $R$ is a {\em q-hypersurface} ring if 
for each finitely generated module $M$ there is a self-map 
$f=f_M: \Sigma^n M \lra M$  so that $M/f$ non-trivial and  small.
\end{defn}

\begin{lemma}
For a complete local Noetherian ring $R$, the conditions hypersurface, 
g-hypersurface, p-hypersurface 
and z-hy\-per\-sur\-face
are all equivalent. They imply the q-hypersurface condition.
\end{lemma}

\begin{proof}
 We have 
$$\mbox{hypersurface} \Rightarrow 
\mbox{p-hypersurface} \Rightarrow \mbox{g-hypersurface}; $$
the first of these is proved in Section \ref{sec:hypersurfaces1} 
 and the second is a tautology. 
Avramov \cite[8.1.1.3]{Avramov:1998a} shows that
$$\mbox{g-hypersurface} \Rightarrow \mbox{hypersurface},$$
so that the first three conditions are equivalent. 

We have
$$\text{\rm hypersurface} \Rightarrow 
\text{\rm z-hypersurface} \Rightarrow \text{\rm q-hypersurface}; $$
the first of these is proved in Section \ref{sec:hypersurfaces1} 
 and the second is clear from the construction since the operator is of
non-zero degree.  

We show  in  Remark \ref{rem:zhypersurfaceisphypersurface} 
that 
\begin{equation*}
\text{\rm z-hypersurface} \Rightarrow \text{\rm p-hypersurface}. 
\qedhere
\end{equation*}
\end{proof}

\section{The Jacobson radical of $ZD(R)$.}
\label{sec:jac}
We give partial information on the Jacobson radical of the centre of $D(R)$ 
sufficient to let us show that the derived category of a complete intersection 
is indeed zci. 

\subsection{Degree and the Jacobson radical.}
There is one easily verified condition for membership of the 
Jacobson radical.
 
\begin{prop}
\label{prop:nonzeroisjac}
If $R$ is a commutative ring then any element of $ZD(R)$ of non-zero
degree lies in the Jacobson radical of $ZD(R)$.
\end{prop}

\begin{proof}
We first prove that an element $\chi$ of negative codegree is in the Jacobson 
radical. Start by noting that $\chi$  is zero on modules $M$ since negative
Ext groups are zero.  Now we argue by induction on $l$ that $\chi$ is nilpotent
on complexes $X$ of length $l$. Indeed, if $X$ is of length $l+1$ we may take
$Y$ to be the subcomplex omitting the top nonzero entry of $X$, and we have
the short exact sequence 
$$0 \lra Y \lra X \lra \Sigma^d M \lra 0$$
where the top of $X$ is $M$ in degree $d$. The map $\chi$ induces a map 
of the resulting triangle. It is zero on $M$, and $\chi^n=0$ on $Y$ for some
$n$ by induction. Thus $\chi^n:X \lra X$ lies in a triangle with two zero 
maps, and is thus nilpotent.

 Thus for any $y \in ZD(R)$ the element
  $1-\chi y$ acts as a unit on all bounded complexes. Consider it as a map
of an arbitrary complex $Z$; since any homology class is supported on a finite
subcomplex the induced map on homology is an isomorphism, and $1-\chi y$ is
invertible on $Z$. In more detail, to see that $1-\chi y$ is surjective on 
homology, pick a homology class $\alpha \in H_*(Z)$, and then a bounded
subcomplex $Z'$ of $Z$ with $\alpha =i_*\alpha'$ where $i : Z' \lra Z$ 
is the inclusion. Since $1-\chi y $ is an isomorphism on $Z'$, we may choose
$\tilde{\alpha}'$ so that $\alpha'=(1-\chi y)(\tilde{\alpha}')$. Now
$\alpha =(1-\chi y) (i_*(\tilde{\alpha}'))$. The proof of injectivity is
similar. 

Now if $\chi $ is of positive codegree, and $1-\chi y$ is homogeneous of degree
0, then $y$ is of negative codegree and hence in the Jacobson radical by the
above argument. Thus $1-\chi y$ is a unit as required. 
\end{proof}

\subsection{Bimodules and the centre.}
 Bimodules provide a useful source of elements of  
$ZD(R)$. Indeed, if $R$ is a flat $l$-algebra, and if $X \lra Y$ is a
map of  $R$-bimodules 
over $l$ (which is to say, of modules over $R^e=R \tensor_l R$), then 
for any $R$-module $M$ we obtain a map 
$$X \tensor_R M \lra Y \tensor_R M$$
of $R$-modules, natural in $M$. 

It is convenient to package this in terms of the Hochschild cohomology 
ring 
$$HH^*(R|l)=\Ext_{R^e}^*(R,R). $$
If $l=\Z$ it is usual to omit it from the notation. 
Thus we obtain elements of the centre
from Hochschild cohomology in the form of  a ring homomorphism
$$HH^*(R|l) \lra ZD(R). $$
If $R$ is an $l$-algebra which is not flat, $R^e=R \tensor_l R$ is
taken in the derived sense, and similarly for $HH^*(R|l)$. 

Given maps
$l\lra l' \lra R$, we obtain a map $R\tensor^lR \lra R\tensor^{l'}R$
and hence a ring map $HH^*(R|l')\lra HH^*(R|l)$. In particular, we have
maps
$$R=HH^*(R|R) \lra HH^*(R|l) \lra HH^*(R|\Z)=HH^*(R).$$

\subsection{Hochschild cohomology transcended.}
\label{subsec:transcend}

It seems natural to relax the role of  Hochs\-child cohomology.
If the $R^e$-module $R$ is virtually small, this will suffice for 
most purposes. Indeed, we  may suppose $R\finbuilds_{R^e} W$
with $W$ small and non-trivial. Given any $R$-module $M$ we may 
apply the functor $\tensor_R M$ to the building process, to find
$$M=R\tensor_R M \finbuilds_R W\tensor_R M.$$
On the other hand, since
$$R^e\finbuilds_{R^e} W,$$
we conclude
$$R\tensor_l M \simeq R^e\tensor_R M\finbuilds_{R} W\tensor_R M.$$
If $M$ is small over $l$ (for example if $l$ is regular and $M$ 
is finitely generated over $l$) then  $R\tensor_l M$ is 
small, and it follows that $W\tensor_RM$ is small. 

This suggests defining $R$ to be {\em bimodule ci (bci)} if 
$R$ is a virtually small bimodule over $l$ for some regular 
ring $l$ over which $R$ is small. 

\begin{remark} The relationship between bci and qci deserves
further investigation.
If $W$ is an $R^e$-module, we may consider the 
class $V(W)$ consisting of $R$-modules $M$ so that $M\tensor_RW\simeq 0$.
This is a localizing subcategory of $D(R)$.  
For any virtually small $R^e$-module $X$ we may then 
consider the collection $I(X)$ of all localizing
subcategories of the form $V(W)$ as $W$ runs through small $R^e$-modules
finitely built by $X$. If $I(X)$ contains a localizing subcategory $V(W)$
intersecting $\fg$ in 0, then bci implies qci. 
\end{remark}

\section{Hypersurfaces.}
\label{sec:hypersurfaces1}

In this section we suppose $R$ is a hypersurface ring, so that
 we have a short exact sequence
$$0 \lra Q \stackrel{f}\lra Q \lra R \lra 0$$
of $Q$-modules for a regular local ring $Q$.
 There are two basic constructions that we need to generalize. 

%Readers for whom this is familiar may enjoy the exercise of extending
%this to graded rings $R$ and $Q$, permitting $f$ to be an element of 
%degree $d$. 

\subsection{The degree 2 operator.}
\label{subsec:operator}

We describe a construction of a cohomological operator due to Gulliksen 
\cite{Gulliksen:1974a,Avramov:1998a}.

Given an $R$-module $M$ we may apply $(\cdot) \tensor_QM$ to the defining
sequence to obtain the short exact sequence
$$0 \lra \Tor_1^Q(R,M) \lra M \stackrel{f}\lra M \lra R\tensor_Q M \lra 0.$$
Since $f=0$ in any $R$-module, we conclude
$$\Tor_1^Q(R,M)\cong M, R\tensor_Q M \cong M, $$
and 
$$\Tor_i^Q(R,M)=0 \mbox{ for } i \geq 2.$$
Next, choose a free $Q$-module $F$ with an epimorphism to $M$, giving a 
short exact sequence 
$$0 \lra K \lra F \lra M \lra 0$$
of $Q$-modules. Applying $(\cdot)\tensor_QR$, we obtain 
$$0 \lra \Tor_1^Q(R,M) \lra \Kbar \lra \Fbar \lra M \lra 0.$$
Since $\Tor_1^Q(R,M) \cong M$, this gives an element 
$$\chi_f \in \Ext^2_R(M,M), $$
or a map 
$$\chi_f: M \lra \Sigma^2M$$
in the derived category of $R$-modules.

We will see in Subsection \ref{subsec:chiinHochschild}
 that this construction lifts to give an element 
$$\chi_f \in HH^2(R|Q), $$
and hence in particular that it gives an element of the centre
$ZD(R)$ of $D(R)$.

\subsection{The eventually periodic resolution.}
\label{subsec:eventperiod}
Continuing with the above case, we may show that all modules $M$ have
free resolutions over $R$ which are eventually periodic of period 2. 

Indeed, $M$ has a finite free $Q$-resolution  
$$0 \lra F_n \lra F_{n-1} \lra \cdots \lra F_1 \lra F_0 \lra M \lra 0. $$
Adding an extra zero term if necessary, we suppose for convenience that $n$ is even.
Now apply $(\cdot) \tensor_Q R$ to obtain a complex
$$0 \lra \Fbar_n \lra \Fbar_{n-1} \lra \cdots \lra \Fbar_1 \lra \Fbar_0 \lra M \lra 0. $$
Since $\Tor_i^Q(R, M)=0$ for $i \geq 2$, this is exact except in homological 
degree 1, where it is $\Tor_1^Q(R,M)\cong M$. Splicing in a second copy of the resolution, 
we obtain a complex
\[ \addtolength{\arraycolsep}{-0.4mm}
\begin{array}{ccccccccccccccccc}
0           &\lra &
\Fbar_n     &\lra & 
\Fbar_{n-1} &\lra & 
     \cdots &\lra &
\Fbar_2     &\lra &         
\Fbar_1     &\lra &
\Fbar_0     &\lra &
M           &\lra &
0 \\
\oplus      &&
\oplus      &&
\oplus      &&
      &&
\oplus      &\nearrow&
      &&
      &&
      &&\\

\Fbar_{n-1} &\lra &
\Fbar_{n-2} &\lra & 
\Fbar_{n-3} &\lra & 
     \cdots &\lra &
\Fbar_0      & &         
            &       &
            &      &
           &      &
\end{array} \]
which is exact except in the second row, in homological degree 4, 
where the homology
is again $M$. We may repeatedly splice in additional rows to 
obtain a free resolution
$$\ldots \lra G_3 \lra G_2 \lra G_1 \lra G_0 \lra M \lra 0$$
over $R$. Remembering the convention that $n$ is even, provided  the degree is 
at least $n$,  the modules in the resolution are 
\[ G_{2i}=\Fbar_{n} \oplus \Fbar_{n-2} \oplus \cdots \oplus 
\Fbar_2 \oplus \Fbar_0 \]
in even degrees and 
\[ G_{2i+1}=\Fbar_{n-1} \oplus \Fbar_{n-3} \oplus \cdots 
\oplus \Fbar_3 \oplus \Fbar_1 \]
in odd degrees. 

\subsection{Smallness.}

We may reformulate the eventual periodicity of the previous subsection in 
homotopy invariant terms. 

\begin{lemma}
If $R$ is a hypersurface $R=Q/(f)$ and $M$ is a finitely generated $R$-module
 then the mapping cone of 
$\chi_f: M \lra \Sigma^2M$ is small. 
\end{lemma}

\begin{proof}
From the Yoneda interpretation, we notice that 
$$\chi_f : M \lra \Sigma^2M$$
is realized by the quotient map factoring out the first row subcomplex
$$0 \lra \Fbar_n \lra \Fbar_{n-1} \lra \cdots \lra \Fbar_1 \lra \Fbar_0\lra 0. $$
Thus the short exact sequence
$$0 \lra \Fbar_{\bullet} \lra G_{\bullet} \lra \Sigma^2 G_{\bullet} \lra 0$$
of $R$-free chain complexes realizes the triangle
\begin{equation*}
\Sigma^{1}M/\chi \lra M \lra \Sigma^2M.
\qedhere
\end{equation*}
\end{proof}

\begin{remark}
\label{rem:zhypersurfaceisphypersurface}
Note that if there is a map $\chi : M \lra \Sigma^2M$ whose fibre $F$ 
is small, then we may form an eventually 2-periodic resolution of $M$. Indeed, 
we may suppose $F$ is a finite complex of finitely generated projectives, 
and then realize the composite $\Sigma F \lra \Sigma M \lra F$
by a map $\delta $ of complexes. The double complex formed from 
$\bigoplus_{n \geq 0}\Sigma^n F$ by using $\delta$ to relate the 
successive copies of $F$ is equivalent to $M$. It is 2-periodic once $n$ 
is larger than the span of $F$. 
\end{remark}

\subsection{Lifting to Hochschild cohomology.}
\label{subsec:chiinHochschild}
We now explain how the degree 2 operators can be lifted to 
Hochschild cohomology classes, working here in the underived context.
We assume that $Q$ and $R$ 
are flat $l$-algebras, and use the standard notation 
$Q^e=Q \tensor_l Q, R^e=R \tensor_l R,$ and 
the formula 
$$HH^*(R|l)=\Ext_{R^e}^*(R,R)$$
for the  Hochschild cohomology. 

Next, note that there is an exact sequence
$$0 \lra Q\tensor_l R \stackrel{f\tensor 1 }\lra  Q\tensor_l R \lra R^e \lra 0$$
so that 
$$\Tor_1^{Q\tensor_l R}(M,R^e)\cong M. $$
Now define $J$ by the exact sequence
\begin{equation}
\label{Jeqn}
0 \lra J \lra Q\tensor_l R \lra R \lra 0
\end{equation}
of $(Q,R)$-bimodules. We obtain the required relationship by 
forming 
$$R \tensor_Q (\ref{Jeqn}) \tensor_R M$$
in two different ways. 

First we apply $R \tensor_Q (\cdot)$ 
%$$R \tensor_Q (\cdot) =(\cdot) \tensor_{Q\tensor_l R}R^e$$ 
to obtain an exact sequence
\begin{equation}
\label{RJeqn}
0 \lra R \lra \Jbar \lra R^e \lra R \lra 0 
\end{equation}
of $R^e$-modules: this is the element $\chi_f \in HH^2(R)$.

To see that $\chi_f$ gives the required endomorphism of the identity 
functor on $D(R)$ we proceed as follows. For any $R$-module $M$, we can apply 
$(\cdot ) \tensor_R M$ to Equation (\ref{RJeqn}) to obtain an exact sequence
 $$0 \lra M \lra \Jbarbar \lra R \tensor_lM \lra M \lra 0 .$$
It remains to observe that this can also be obtained from 
an exact sequence of $Q$-modules by tensoring down. 
However for this we need only apply $(\cdot )\tensor_R M$ 
to Equation (\ref{Jeqn}) to obtain 
$$0 \lra \Jbar \tensor_R M \lra Q \tensor_lM \lra M \lra 0, $$
which qualifies as a sequence
$$0 \lra K \lra F \lra M \lra 0$$
as in Subsection \ref{subsec:operator}.

\section{General complete intersections.}
\label{sec:ciiszci}
Building on the hypersurface case, we can deal with general complete
intersections. In the general case, we suppose $R=Q/(f_1, \ldots ,f_c)$ where
$Q$ is a regular local ring and $f_1, f_2, \ldots, f_c$ is a regular sequence.

As a matter of notation, let $Q=Q_0$, and $Q_i=Q/(f_1, \ldots, f_i)$, so 
that $Q_c=R$  and we have a sequence of rings
$$Q=Q_0 \lra Q_1\lra \cdots \lra Q_c=R. $$
The contents of Subsection \ref{subsec:operator} only required
$f$ to be regular and made no requirement on $Q$. We deduce
that for any $Q_{s+1}$-module $M_{s+1}$ we have
$$\Tor_i^{Q_s}(Q_{s+1},M_{s+1})=
\trichotomy
{M_{s+1} \mbox{ if } i=0}
{M_{s+1} \mbox{ if } i=1}
{0 \mbox{ if } i\geq 2}.$$
We have an operator $\chi_j=\chi_{f_j}$ on $Q_j$-modules. 

\begin{prop}
For any finitely generated $R$-module $M$ the 
 iterated mapping cone 
$$ M/\chi_c/\chi_{c-1}/\cdots /\chi_1$$
is finitely built from $M$ and small over $R$. 
\end{prop}

\begin{proof}
In the hypersurface case, we constructed an explicit eventually 
periodic resolution. That is to say, in large degrees the resolution
is obtained by splicing together a fixed resolution. 
In general the analogue of a periodic resolution is a multigraded
resolution obtained by stacking boxes. For complete intersections, 
 one may build an explicit, resolution which is multigraded and
eventually obtained by stacking boxes. 

Suppose $M$ is an $R$-module. We begin with the construction of Subsection 
\ref{subsec:eventperiod}. This gives a $Q_0$-resolution
$$0 \lra F_n \lra F_{n-1} \lra \cdots \lra F_1 \lra F_0 \lra M \lra 0$$
of $M$, which we tensor down to give a $Q_1$-complex
$$0 \lra \Fbar_n \lra \Fbar_{n-1} \lra \cdots \lra \Fbar_1 \lra \Fbar_0 \lra M \lra 0 .$$
 This is not exact, so we splice together copies until we
obtain a free $Q_1$-resolution 
$$ \cdots \lra G_2 \lra G_1 \lra G_0 \lra M \lra 0 ,$$
and there is a short exact sequence of $Q_1$-chain complexes
$$\Fbar_{\bullet} \lra G_{\bullet} \lra \Sigma^2G_{\bullet} $$
realizing
$$\Sigma^{1} M/\chi_1 \lra M \lra \Sigma^2 M .$$
Thus 
$$M/\chi_1 \simeq \Sigma^{-1}\Fbar_{\bullet}$$ 
is $Q_1$-small.

We may repeat the process, applying
$(\cdot)\tensor_{Q_1}Q_2$ to give a $Q_2$-complex
$$ \cdots \lra \Gbar_2 \lra \Gbar_1 \lra \Gbar_0 \lra M \lra 0 .$$
 Since $\Tor_i^{Q_1}(Q_2,M)=0$ for $i \geq 2$
it is exact except in degree 2, where it is $\Tor_1^{Q_1}(Q_2,M)\cong M$, 
and we may splice together complexes to give a free $Q_2$-resolution 
$$ \cdots \lra H_2 \lra H_1 \lra H_0 \lra M \lra 0 .$$
There is a short exact sequence of chain complexes
$$\Gbar_{\bullet} \lra H_{\bullet} \lra \Sigma^2H_{\bullet} $$
realizing
$$\Sigma^{1} M/\chi_2 \lra M \lra \Sigma^2 M .$$
Thus 
$$M/\chi_2 \simeq \Sigma^{-1} \Gbar_{\bullet},$$
and therefore 
$$M/\chi_2/\chi_1 \simeq \Sigma^{-2} \Fbarbar_{\bullet}$$
is $Q_2$-small. 

Repeating this process $c$ times, we obtain the desired conclusion.
\end{proof}

\section{Growth conditions.}
\label{sec:zciisgci}
Throughout algebra and topology it is common to use the rate of growth of
homology groups as a measurement of complexity. In this section we show
that if $R$ satisfies the zci condition then its 
Ext algebra has polynomial growth (i.e., $R$ satisfies the gci condition).
Gulliksen's theorem then completes the proof of equivalence of the 
characterizations of complete intersections. Because of later applications
we will prove results in greater generality than we require here.

\subsection{Polynomial growth.}

Often, one knows that 
the homology groups will be zero in negative degrees, so the growth 
condition will be on the positive homology groups. Similarly, if homology 
 groups are zero in positive degrees,  the growth 
condition will be on the negative homology groups. For us it is convenient
to treat positive and negative degrees symmetrically. 

To continue, suppose that $H$ is a Noetherian graded ring and
that $M$ is an $H$-module (we have in mind the case the $H$ is the homology 
of a DGA, and $M$ is the homology of a DG-module over it). 
If $H_0$ is local, we write $\len (M_i)$ 
for the length of $M_i$. To extend the notion to more general 
rings we impose a condition to give a reasonable notion of length.

\begin{defn}
(i) An $H$-module $M$ is {\em locally finite} if $M$ is a Noetherian 
$H$-module and
 $\len ((M_i)_{\fm_0})$ is finite for all $i$ and all maximal ideals $\fm_0$ 
of $H_0$.

(ii) If $M$ is locally finite then since it is Noetherian, 
it has finitely many associated primes $\wp$, and we may write
$$\len (M_i)=\sum_{\wp\in \ass (M)}\len((M_i)_{\wp_0}).$$
\end{defn}

Assembling these we obtain a notion of growth.

\begin{defn}
We say that a locally finite module $M$ has polynomial growth of 
degree $\leq d$ and write
$\growth (M) \leq d$ if 
there is a polynomial $p(x)$ of degree $d$ with  
$$\len (X_n)\leq p(|n|)$$
for almost all $n$.
\end{defn}

\begin{remark}
A bounded module has growth $\leq -1$. For
modules with growth $d\geq 0$, by adding a constant to the
polynomial, we may insist that the
bound applies for all $n \geq 0$.
\end{remark}

\subsection{Degree 0 elements are redundant.}
We show that degree 0 elements of the Jacobian do 
not reduce the growth rate. 

\begin{lemma}
\label{lem:degzeronugatory}
If $x \in \Jac (H)$ and $|x|=0$ then for any locally finite $H$-module $M$ 
$$\growth (M/xM)=\growth (M).$$
\end{lemma}

\begin{proof}
Since $M$ is Noetherian it has finitely many associated primes $\wp$. 
If $\wp$ is an associated prime, 
the degree 0 part $\wp_0$ is maximal in $H_0$, since $\len (M_i) <\infty$.  
Altogether, we suppose $\fm_0^1, \ldots , \fm_0^S$ are the finitely many 
maximal primes of $H_0$ occurring in this way. We thus have
$$\len (M_i)=\sum_s \len ((M_i)_{\fm_0^s}), $$
and hence
$$\growth (M_i) =\max\{ \growth( (M)_{\fm_0^s})\st s=1,\ldots , S\} $$
In short, we may assume that $R_0$ is local, and we write $\fm_0$ for its
maximal ideal.

Next, since $\ann (M)$ is finitely generated, $\ann (M)_0$ is $\fm_0$-primary, 
and therefore there is an $N$ so that $\fm_0^N \cdot M_i=0$ for all $i$.

Now 
$$\len (M_i)\leq \len (\fm_0/\fm_0^N)\cdot \len (M_i/\fm_0M_i), $$
and hence
$$\len (M_i/xM_i)\geq \len (M_i/\fm_0 M_i)\geq \frac{\len (M_i)}
{\len(\fm_0/\fm_0^N)}. $$
This shows $\growth (M/xM)\geq \growth (M)$ and completes the proof. 
\end{proof}

\subsection{Mapping cones and growth rate.}
We now turn to the special case when we have a DGA $R$, and
a DG-module $X$ over $R$, and we take $H=H_*(R)$ and 
consider the growth of $M=H_*(X)$. Of course, we are concentrating 
on the special case
when $R$ is itself concentrated in a single degree, so that
this means $H=R$ and $X$ is nothing but a chain complex.

\begin{defn}
We say that $X$ is {\em locally finite} over $R$ if $H_*(X)$ is 
locally finite over $H_*(R)$, 
and that $X$ has growth $\leq d$ if this is true for $H_*(X)$. 
\end{defn}

The key link between the zci condition and growth is the following
result. 
\begin{lemma}
Given bounded below locally finite modules $X,Y$ in a triangle 
$$\Sigma^n X \stackrel{\chi}\lra X \lra Y$$
with $\chi \in \Jac (ZD(R))$
$$\growth (X) \leq \growth (Y)+1.$$
\end{lemma}

\begin{proof}
By the arguments of the previous subsection, we may assume $n\neq 0$. 
Indeed if $\chi$ has degree 0 we have an exact sequence
$$0\lra H_*(X)/\chi \lra H_*(Y) \lra H_*(X), $$
so that Lemma \ref{lem:degzeronugatory} shows that 
$$\growth (X)=\growth (Y). $$

Next, since $X$ and $Y$ are locally finite, the homology modules
$H_*(X)$ and $H_*(Y)$  have only 
finitely many associated primes
and it suffices to work with one maximal ideal at a time. Accordingly
we may localize and it suffices to deal with the case that 
$H_0(R)$ is a local ring. 

The homology long exact sequence of the triangle includes 
$$\cdots \lra H_{i-n}(X) \stackrel{\chi}\lra 
H_i(X) \lra H_i(Y)\lra \cdots .$$
This shows
$$\len (H_i(X))\leq \len (H_i(Y))+\len (\chi H_{i-n}(X)) .$$
Iterating $s$ times, we find
$$\len (H_i(X))\leq \len (H_i(Y))+\len(\chi H_{i-n}(Y)) +\cdots +
\len(\chi^{s-1} H_{i-(s-1)n}(Y)) +
\len(\chi^s H_{i-sn}(X)) .$$

First suppose that $n>0$.
The case that $H_*(X)$ is bounded below is familiar.
If $h_X(t)$ is the Hilbert series of $H_*(X)$ then we have
$$h_X(t) \leq h_Y(t)(1+t^{n+1}+t^{2n+2}+\cdots )=\frac{h_Y(t)}{1-t^{n+1}},$$
giving the required growth estimate. 

To extend to the general case, we argue that $\chi^s H_{i-sn}(X)=0$ for $s>>0$. 
Indeed, we have a decreasing chain of subgroups of the finite length module 
$H_i(X)$, which therefore stabilizes. The stable value must be zero by Nakayama's lemma. Thus there is a $j=j(i)$ so that if $i-sn \leq j(i)$ then $\chi^s$ 
is zero from degree $i-sn$. Now taking $j=\min\{j(0),j(1), \ldots , j(n-1)\}$, 
we see that for $k\leq j$ powers of $\chi$ from $H_k(X)$ have zero image
in positive degrees. This means that the number $j$ serves as a lower bound
for estimates in positive degrees, and we can use the familiar argument.

For estimates in negative degrees we argue similarly, starting 
with the estimate
$$\len (H_j(X))\leq \len (H_{j+n+1}(Y))+\len (\chi H_{j}(X)) .$$
This gives the bounded above case in the usual way. Using Nakayama's
lemma we find as $s$ so that $\chi^{s}H_j(X)=0$, and deduce that 
there is a $k$ so that no power of $\chi$ from negative degrees has
non-zero image in degrees $\geq k$. The general case then follows
like the bounded above case. 

The argument if $n<0$ is precisely similar. 
%Finally, if $n=0$, we see as before that for each $i$ there 
%is a number $s(i) $ so that $\chi^{s(i)}H_i(X)=0$. However if
%$n=0$ there is a  single number $s$ so that we may take $s(i)=s$ for 
%all $i$. The existence of   $s$ follows since $H_*(X)$ is a 
%Noetherian module, so that  the chain of submodules $\ann (\chi^{s},H_*(X))$
%is eventually constant. It then follows that 
%$h_X(t) \leq sh_Y(t)$, giving a stronger growth estimate than required. 
\end{proof}

%\begin{remark}
%We understand Avramov-Iyengar have verified that the same applies to 
%graded local rings. 
%\end{remark}

\subsection{Operators and growth.}

It is convenient to introduce further terminology
in which the letter `g' stands for growth.
\begin{defn}
If $R$ is a commutative local ring we say that it is {\em gci} of codimension 
$c$ if $\Ext_R^*(k,k)$ has polynomial growth of degree $c-1$.
\end{defn}

\begin{thm} (Gulliksen \cite{Gulliksen:1971a,Gulliksen/Levin:1969a})
\label{thm:Gulliksen}
A Noetherian local ring is a complete intersection if and only if
it is gci.\qed
\end{thm}

We may connect this to our module theoretic condition.

\begin{thm}
If $R$ is zci of codimension $c$ then it is also gci of codimension $c$.
\end{thm}

\begin{proof}
By hypothesis there are elements $\chi_1,\chi_2, \ldots , \chi_c$ 
in $\Jac (ZD(R))$ so that $k/\chi_c/\ldots /\chi_1$ is small.

We may apply $\Hom_R(\cdot , k)$ (i.e., derived Hom) 
to the resulting triangles in the
derived category. For brevity, we write 
$$X_i=\Hom_R(k/\chi_c/\cdots / \chi_i, k). $$
Now by hypothesis, since $k/\chi_c/\ldots /\chi_1$ is small, so
$X_1$ is bounded, and thus of growth $-1$, and 
$$H_*(X_{c+1})=H_*(\Hom_R(k,k))=\Ext_R^*(k,k), $$
so that  the theorem states
$\growth (X_{c+1})\leq c$.

Applying the lemma to the cofibre sequence
$$\Sigma^{-|\chi_{i-1}|}X_i\lla X_i\lla X_{i-1}$$
we find 
$$\growth (X_i)\leq \growth (X_{i-1})+1$$
and hence, repeating $c+1$ times, we find
$$\growth (X_{c+1})\leq c$$
as required. 
\end{proof}

\begin{remark}
In commutative algebra one says that a locally finite complex $M$ of
$R$-modules has {\em complexity} $n$ if $\Ext_R^*(M,k)$ has polynomial growth 
of degree $n-1$. In these terms, the proof proceeds by showing
for $i=1,2,\ldots , c+1$
that $k/\chi_c/\cdots /\chi_{i+1}$ has complexity $i$. 
\end{remark}

\begin{cor}
The three conditions ci, zci and gci on Noetherian local rings are 
equivalent.
\end{cor}

\begin{proof}
We showed in Section \ref{sec:ciiszci}
that ci implies zci, and in Section \ref{sec:zciisgci} that zci implies gci.
Gulliksen's Theorem (\ref{thm:Gulliksen}) shows   gci implies ci.
\end{proof}

%\bibliographystyle{amsplain}
%\bibliography{../repcoh}

\end{document}